\documentclass[11pt, a4paper]{article}

\usepackage{amsmath, amsthm, amsfonts, amssymb}

\usepackage{setspace}
\usepackage{fullpage}
\usepackage{enumitem}

\usepackage{floatpag}
\usepackage[dvipsnames]{xcolor}

\usepackage{array}
\usepackage{float}
\usepackage{caption}
\usepackage{subcaption} 

\usepackage{hyperref}
\hypersetup{colorlinks=true,
	citecolor=blue,
	filecolor=blue,
	linkcolor=blue,
	urlcolor=blue
}
\usepackage{cleveref}
\usepackage{cite}

\usepackage{comment}

\newtheorem{theorem}{Theorem} 
\newtheorem{corollary}{Corollary} 
\newtheorem{definition}{Definition}

\newtheorem{question}{Question}

\newtheorem{proposition}{Proposition}
\newtheorem{conjecture}{Conjecture}

\begin{document}
	
	\title{A note on multicolour  Erd\H{o}s-Hajnal conjecture}
	\date{\vspace{-5ex}}
	\author{ Maria Axenovich
		\thanks{ Corresponding author;
			Karlsruhe Institute of Technology, Karlsruhe, Germany; 
			email:
			\mbox{\texttt{maria.aksenovich@kit.edu}}.
		}     
\and
		Lea Weber
		\thanks{
			Karlsruhe Institute of Technology, Karlsruhe, Germany;
			email: \mbox{\texttt{0815lea@gmx.net}}.
	}}

	\maketitle

\abstract{Informally, the Erd\H{o}s-Hajnal conjecture  (shortly EH-conjecture) asserts that if a sufficiently large host clique on $n$ vertices is edge-coloured avoiding a copy of some fixed edge-coloured clique, then there is a large homogeneous set of size $n^\beta$ for some positive $\beta$, where a set of vertices is homogeneous if it does not induce all the colours.  This conjecture, if true, claims that imposing local conditions on edge-partitions of cliques results in a global structural consequence such as a large homogeneous set, a set  avoiding all edges of some part.  While this conjecture attracted a lot of attention, it is still open even for two colours.  \\

In this note, we reduce the multicolour EH-conjecture to the case when the number of colours used in a host clique is either the same as in the forbidden pattern or one more.  We exhibit a non-monotonicity behaviour of homogeneous sets in coloured cliques with forbidden patterns by showing that allowing an extra colour in the host graph could actually decrease the size of a largest homogeneous set. \\

}

\section{Introduction}
We shall be considering edge-coloured cliques.  Here,  a {\it clique} of size $n$   is a complete graph on $n$ vertices denoted $K_n$.   A {\it  co-clique} of size $n$  is a graph on $n$ vertices with no edges.  For a graph $H$, we denote by  $\alpha(H)$ and $\omega(H)$ the sizes of a largest induced co-clique and clique in $H$, respectively.
An \emph{$s$-edge-colouring} $c$ of the  clique $K_n$ on vertex set $[n]$ is a map $c: \binom{[n]}2 \to [s]$. We denote by $|c|$ the number of colours from $[s]$ for which $c^{-1}$ is not empty.
Note that an $s$-edge-colouring $c$ of $K_n$ can be seen as an edge-partition of $G$ into $s$ colour classes, i.e.\ $K_n = G_1 \cup \cdots  \cup G_s$, where $G_i$ corresponds to a maximal subgraph of $K_n$ whose edges are assigned colour $i$ under $c$. Here $G_i$ can be an empty graph if $|c| < s$.  For an $s$-edge-colouring $c$ of $K_n$  on vertex set $V$ and an $s'$-edge-colouring $c'$ of $K_k$, we  define  a {\it copy } of $c'$ in $c$ to be a clique  on a set $U \subseteq V$ of size $k$, such that $c$ restricted to this clique is isomorphic to $c'$, i.e., there is a bijection $\phi: U\rightarrow [k]$ so that for any two vertices $x, y\in U$, $c(xy) = c'(\phi(x)\phi(y))$.  We say that $c$ is $c'$-free if there is no copy of $c'$ in $c$.   Typically we assume that $k$ is fixed and $n$ is large, i.e.\ the $c'$-free property is a local condition on the colouring. One can think of the colouring $c'$ as a forbidden colour pattern. One of the key questions considered is how the local restrictions impact global properties, in particular how large the homogeneous number must be.\\

A \emph{homogeneous set} in an $s$-edge-colouring $c$ of $K_n$ is a set $X \subseteq [n]$ that has a colour ``missing'', i.e.\ $|\{c(xy):   x,y \in X\}|< s$. The size of a largest homogeneous set in $c$ is denoted by $h_{s}(c)$  or if the number of colours is clear from the context, simply $h(c)$. Note that any homogeneous set is an independent set in some colour class $G_i$, $i \in [s]$. Thus, we have $h(c) = \max\{\alpha(G_i): {i \in [s]} \}$. Note that  for $s=2$ one colour of $c$ corresponds to the edges of some $n$-vertex graph $G$ and the other colour corresponds to the edges of  the complement of $G$. Then in particular, we have $h(c) = h(G) = \max \{\alpha(G), \omega(G)\}$, which coincides with the definition of a homogeneous set in graphs. For $s'\leq s$, an $s'$-edge-colouring $c'$ of $K_k$, $k \le n$, we define 
$$h(n, c') = h_{s}(n, c')=\min\{h_s(c)\ |\ \text{$c$ is a $c'$-free $s$-edge-colouring of $K_n$}\}.$$

\begin{definition}
	Let $c'$ be an $s'$-edge-colouring of $K_k$ and let $s \ge s'$. If there are positive constants $\epsilon = \epsilon(c', s)$ and $C$, such that $h_s(n, c')\geq Cn^\epsilon$,
	we say that $c'$ has the EH-property for $s$ colours. 
\end{definition}

For example, when $c'$ is an edge-colouring of $K_3$,  i.e. a triangle,  with two edges of  colour $1$ and one edge of colour $2$, one can show that in any $c'$-free edge-colouring of a clique on $n$ vertices using colours $1, 2, $ and $3$ there is a set of vertices of size $n^{1/2}$ inducing edges only of two of these three colours, see for example Axenovich, Snyder, Weber~\cite{ASW}.  This shows that $c'$ has the EH-property for $3$ colours.

\begin{conjecture}[Erd\H{o}s,  Hajnal~\cite{EH89}] \label{conj:EH}
	Let $k, s'$ be integers with $k, s' \ge2$. Then for any $s\ge s'$, any edge-colouring $c'$ of $K_k$ with $|c'| = s'$ has the EH-property for $s$ colours.
\end{conjecture}

Even in the case of two colours, i.e., $s'=s=2$,  the conjecture remains open, see for example  a survey by Chudnovsky~\cite{chudnovsky14}, as well as \cite{APS, BLT, FPS, NSS}, to name a few.   When $F$ is a fixed graph and $G$ is any $F$-free $n$-vertex graph, Erd\H{o}s and Hajnal proved that $h(G) \ge 2^{c\sqrt{\log n}}$.  One could extend the arguments in~\cite{EH89} from two colours to $s$ colours to show that 
in the above setting $h_s(n, c') = \Omega(2^{\sqrt{ \log n}})$.  In particular, for any $s$ and any forbidden colouring $c'$, we have for any $y$ and any sufficiently large $n$, that 
\begin{equation}\label{general-lb}
h_s(n, c') \geq \log^y n.
\end{equation}
The bound for two colours was recently improved  to $h(G) \ge 2^{c\sqrt{\log n\log\log n}}$ by   Buci\'{c}, Nguyen, Scott, and Seymour~\cite{BNSS}.  \\

Here, we consider Conjecture \ref{conj:EH} in general  and note that $c'$ might not use all colours in $[s]$ and it is not immediately obvious whether a larger number of colours in the edge-colouring of the host clique forces larger homogeneous sets. We show that we  can reduce the problem to the case when the number of colours in the edge-colouring of a large clique is the same or one more than the number of colours in the forbidden pattern $c'$. 

\begin{theorem}\label{reduction}
	Let $c$ be an edge-colouring of a clique using exactly $s'$ colours and  let $s$ be an integer with $s > s'$. 
	For any positive $n$,   $h_{s+1}(n, c)\geq h_s(n,c)$ and  for any positive $\xi$,  for sufficiently large $n$,  
	$h_{s+1}(n, c) \leq h_s^{1+\xi}(n, c)$.
	In particular,  $c$ has the EH-property for $s$ colours if and only if $c$ has the EH-property for $s+1$ colours.
\end{theorem}

\begin{corollary}\label{cor:s_or_s+1}
	Let $c'$ be an $s'$-edge-colouring of a clique. Then the EH-conjecture holds for $c'$ if and only if $c'$ has the EH-property for $s'$ and $s'+1$ colours.
\end{corollary}

Note that Theorem \ref{reduction} does not hold if $s=s'$.  Indeed, as we shall show in the next proposition,  there is a colouring $c$ of a clique on $4$ vertices  in $2$ colours, such that $h_3(n, c) = o(h_2(n,c))$.  Moreover,  there is a colouring $c'$ of a clique on three vertices in three colours, such that $h_4(n, c') \neq \omega(h_3(n, c'))$.  
Here, an edge colouring of a graph is {\it rainbow} if it assigns distinct colours to distinct edges.   

\begin{proposition}\label{prop:monotonicity}
Let $c'$ be a rainbow colouring of $K_3$ with colours $1,2$, and $3$. Let $c$ be an edge-colouring of $K_4$ with colours $1$ and $2$ in which each  class forms a path on three edges.  Then 
$$ h_3(n, c') = \Theta\left(n^{1/3}\log^2 n\right) {\rm~~ and ~} ~ h_4(n,c') = O(n^{1/3} \log^2n),$$
$$h_2(n, c) = n^{1/2} ~{ ~\rm and ~} ~ h_3(c) = O(n^{1/3} \log^{7/3} n).$$
\end{proposition}

\vskip 0.5cm

This paper is structured as follows.  We prove \Cref{reduction}, \Cref{cor:s_or_s+1}, and Proposition \ref{prop:monotonicity} in Section~\ref{sec:multi-proofs}.  We state some concluding remarks and open problems in Section~\ref{sec:conclusion}. We omit floors and ceilings when clear from the context.

\section{Proofs of the main results}\label{sec:multi-proofs}

\begin{proof}[Proof of Theorem \ref{reduction}]
	Let $c$ be an edge-colouring of a clique, $s$ be an integer with $s > |c|$, and  $n$ be a sufficiently large integer.\\

		Let $c''$ be a $c$-free $(s+1)$-edge-colouring of $K_n$.   Since $s > |c|$, there exists a colour $a \in [s]$ which is not used in $c$. The  colour $s+1$ is also not used in $c$. Now recolour all edges of colour $s+1$ in $c''$ with colour $a$ and call the resulting colouring $c'''$. Then $c'''$ is an $s$-edge-colouring of $K_n$ which is $c$-free, since the edges having colours from $c$ are the same in $c''$ and $c'''$. Thus, there is a homogeneous set $X$ in $K_n$ under $c'''$ of size at least $h_s(n, c)$. In particular, $X$ avoids some colour $a' \in [s]$ under $c'''$.
		If $a'\neq a$, then $X$ also avoids $a'$ under $c''$. If $a = a'$, then $X$ avoids $a$ and $s+1$ under $c''$. Thus, in any case, $X$ is a homogeneous set under $c''$, so 
		$h_{s+1}(n, c)  \geq |X|\geq h_s(n,c)$.\\

 Let $0< \xi <1$ and  $n$ be sufficiently large.  Assume that $h_{s+1}(n,c) =h^{1+\xi}$, for some $h$. Note that from (\ref{general-lb}) we have  $h>\log^{2/\xi}n$.  We shall show that $h_s(n, c) \geq h$.
	
	Let $c''$ be a $c$-free $s$-colouring of $K_n$ on vertex set $V$.  We want to show that $h(c'') \ge h$. We shall construct an $(s+1)$-edge-colouring $c'''$ of $K_n$ starting with $c''$ as follows: Recolour each edge with the colour $s+1$ with probability $\frac12$, and leave the colour from $c''$ with probability $\frac12$. Since the  colour $s+1$ is not used in $c$, the new colouring $c'''$ is $c$-free. Assume that $h(c'') <  h$. 
	Then under $c''$ every subset of  $V$ of size   $h$ contains an edge of each colour  in $[s]$. Using the properties of a random graph $G \in \mathcal G(n, \frac12)$, we see that each vertex set of size $2\log n$ and thus, each vertex set of size  $h$, induces an edge of colour  $s+1$ under $c'''$ with probability close to  $1$.
On the other hand,  since  $h(c'') <  h$, we know that in $c''$ each subset of $V$ of size  $h$ induces an edge of  colour $i$ for each $i \in [s]$.  Thus, using Tur\'an's theorem~\cite{Turan41}, a given subset $V$ of  size $h^{1+ \xi}$ induces  at least 
$x=\Omega\left(\binom{h^{1+\xi}}{2} \frac{1}{h}\right)=\Omega(h^{1+2\xi})$ edges of colour  $i$, for any $i\in [s]$, under the colouring $c''$. 
The probability that all these edges of  colour $i$ are recoloured with colour $s+1$ is at most $(1/2)^x$. Thus, the probability that some subset of $h^{1+\xi}$ vertices misses some colour from $[s]$  under $c'''$ is at most  
$$\binom{n}{h^{1+\xi}}s \left(\frac{1}{2}\right)^x   \leq 2^{h^{1+\xi} \log n - h^{1+2\xi}},$$
i.e., close to zero for large $n$. Therefore, with positive probability, all subsets of $h^{1+\xi}$ vertices induce edges of all colours under $c'''$ and so  $h(c''') <  h^{1+\xi}$, a contradiction to our assumption. \\

As a consequence of the two inequalities on $h_s(n,c)$ and $h_{s+1}(n, c)$, we have that  $c$ has the  EH-property for $s$ colours if and only if $c$ has the EH-property for $s+1$ colours.
		 \qedhere
\end{proof}

\vskip 1cm
\begin{proof}[Proof of Proposition \ref{prop:monotonicity}]

We shall first consider forbidding a rainbow triangle.
Let  $c'$ be an edge-colouring  of $K_3$ in which the edges have colours $1, 2$, and $3$. The structure of $c'$-free $3$-colourings of cliques is known and is called a \emph{Gallai colouring}~\cite{Gallai67, GS04}. It is known that $c'$ has the EH-property for $3$ colours,  see for example \cite{FGP15}, where it is shown that 
$ h_{3}(n, c) = \Theta\left(n^{1/3}\log^2 n\right).$ Next we shall show that $h_4(n,c') =O(n^{1/3} \log^2n)$.\\

We shall consider lexicographic products of  colourings.
If $c^*$ and $c^{**}$ are colourings, where $c^*$ colours a clique on vertex set $\{v_1,\ldots,v_k\}$ and $c^{**}$ colours a clique on $y$ vertices,  then the {\it lexicographic product } of $c^*$ and $c^{**}$, denoted $c^*\times c^{**}$ is a colouring of  a clique on vertex set $X_1\cup \cdots \cup X_k$, where $|X_i|=y$, $i=1, \ldots, k$, the edges induced by $X_i$ are coloured according to $c^{**}$, $i=1, \ldots, k$, and all edges between $X_i$ and $X_j$ are coloured with $c^*(v_iv_j)$, for $ 1\leq i<j\leq k$.  Here, we refer to the $X_i$'s as {\it blobs}.
Let for a colouring $c^*$ of a clique and a subset of colours $I$,   $S^{c^*}_I$  be the size of a largest clique using colours only from $I$.
Note that if $c^*$ and $c^{**}$ are colourings,   then 
$S^{c^*\times c^{**}}_I = S^{c^*}_IS^{c^{**}}_I.$\\

Let $c_i$, $i \in [3]$ be a $3$-edge-colouring of $K_{n^{1/3}}$ using colours from  $[4]-\{i\}$ and satisfying $h_3(c_i) =O(\log n)$, i.e. for some positive $C$ any clique on $C\log n$ vertices induces all three colours from $[4]-\{i\}$. Note that such colourings exist and could be chosen by randomly assigning one of the three colours to each edge uniformly. Also note that $c_i$ is $c'$-free for $i \in [3]$. \\
Let $c_4 = c_1 \times c_2 \times c_3$ be the lexicographic product  of $c_1, c_2$ and $c_3$, it is a $4$-edge-colouring of $K_n$.  This is a construction very similar to one used in~\cite{FGP15}. \\

To verify that $c_4$ is $c'$-free,  consider first $c_2\times c_3$.  Since  $c_2$ and $c_3$ are $c'$-free, we only need to check each triangle
	with two vertices in one blob and one vertex in a different blob of the blow-up of $c_2\times c_3$. Since the edges between two different  blobs all have the same colour, the triangle is not rainbow. Thus, $c_2\times c_3$ is $c'$-free. Similarly, we conclude that $c_4=c_1\times(c_2\times c_3)$ is $c'$-free.\\
	
Next, we shall bound the size of a largest homogenous set in $c_4$.  For a subset $\{i,j, k\} $ of $[4]$, we shall  simply write $ijk$. 
We have that $S_{I}^{c_j}= n^{1/3}$ if $I=[4]-\{j\}$ and $S_{I}^{c_j}= S_{I-\{j\}}^{c_j}=O(\log n)$, if $j\in I$. Then 
\begin{eqnarray*}
		S_{123}^{c_4} & = &S_{123}^{c_1}\cdot  S_{123}^{c_2}\cdot S_{123}^{c_3} = O(\log n)O(\log n)O(\log n),\\
		S_{124}^{c_4} & = &S_{124}^{c_1}\cdot S_{124}^{c_2}\cdot S_{124}^{c_3} =  O(\log n)O(\log n) n^{1/3},\\
		S_{134}^{c_4} & = &S_{134}^{c_1}\cdot S_{134}^{c_2}\cdot S_{134}^{c_3}=  O(\log n)n^{1/3}O(\log n), \quad \text{and}\\
		S_{234}^{c_4}  & = &S_{234}^{c_1}\cdot S_{234}^{c_2}\cdot  S_{234}^{c_3} =  n^{1/3} O(\log n)O(\log n).
	\end{eqnarray*}
Since $h_4(c_4) = \max \{S_{ijk}^{c_4}:  ~\{i, j,k\} \subseteq [4], | \{i, j, k\} |=3\}$, we have that $h_4(n, c') \leq h_4(c_4) =O(n^{1/3}\log^2 n)$.

~\\
Now, we shall consider forbidding induced $P_4$ in two colours.
Let $c$ be a $2$-edge-colouring of $K_4$ in which each  colour class induces $P_4$, a path on $4$ vertices. 
Note that $c$ having the EH-property for $2$ colours is equivalent to $P_4$ having the EH-property. Any $P_4$-free graph  $G$ is a co-graph (see for example~\cite{BD84, CLB81} for properties of co-graphs), which is in particular a perfect graph, i.e., $\omega(G) = \chi(G)$, where $\chi(G)$ is the chromatic number. 
As observed by Erd\H{o}s and Hajnal~\cite{EH89},  if $G$ has $n$ vertices,  $n\leq \alpha(G) \chi(G) = \alpha(G) \omega(G)$. Therefore $h(G) = \max\{\alpha(G), \omega(G)\} \geq n^{1/2}$.
In particular, we have $h_2(n, c) \geq  n^{1/2}$. It is also not difficult to show and is proven in \cite{EH89}, that  $h_2(n, c) \leq  n^{1/2}(1+o(1))$.\\

Next we shall construct a $c$-free colouring $c^*$ of $K_n$ on a vertex set $V$  using colours $1, 2, $ and $3$, for sufficiently large $n$.
By the lower bound on the Ramsey number $R(4,t) = \Omega(t^3/\log^4t)$ by Mattheus and Verstraete~\cite{MV},  there exists a graph $H$ on the vertex set $V$ such that  $\omega(H)< 4$ and  $ \alpha(H) < Cn^{1/3}\log^{4/3}n$, for some positive constant $C$. 
To define $c^*$, let the edges not in $H$ be coloured $3$, and each edge of $H$ be coloured $1$ with probability $1/2$ and $2$ with probability $1/2$.  Note that in this colouring each $K_4$ has an edge of colour $3$, and therefore there is no copy of $c$.
We shall argue that with positive probability $h_3(c^*) = O(n^{1/3} \log^{7/3} n).$
	Letting $q(n) = 8\alpha(H)\log(n)$, we shall show that any set of $q(n)$ vertices induces edges of all three colours under $c^*$. 
	Let $X$ be a fixed set of $q(n)$ vertices. By Tur\'an's theorem~\cite{Turan41},  the number of edges induced by $X$ in $H$ is at least 
	$$ e_X = \frac{1}{\alpha(H)} \binom{q(n)}{2} \ge \frac{q^2(n)}{4\alpha(H)}.$$
	If  $p_X$ is the probability  that $X$ induces only edges of colours $2$ and $3$ in $c^*$ or that $X$ induces only edges of colours $1$ and $3$ in $c^*$, then 
	$$p_X \leq 2\cdot 2^{-e_X} \leq 2\cdot 2^{-q^{2}(n)/4\alpha(H)}.$$
	Using the union bound over all $q(n)$-element subsets of $V$, we have that the probability  that  $c^*$ contains a $q(n)$-vertex set inducing edges of only two colours is at most
	$$
		\binom{n}{q(n)} p_X \leq n^{q(n)} 2^{1 - {q^2(n)}/{4\alpha(H)}} = 2^{8\alpha(H)\log^2 (n) + 1 - {16\alpha(H)\log^2(n)}}  < 1.
$$
	Thus, with positive probability there exists a desired colouring. 
\end{proof}
\noindent We remark that we did not attempt to optimise any of the constants involved.

\section{Concluding remarks}\label{sec:conclusion}

The multicolour Erd\H{o}s-Hajnal conjecture is concerned with the existence of large homogeneous sets in edge-coloured cliques that do not contain a copy of a given colouring of a small clique. It could be that the number of colours used in a large clique is strictly larger than the number of colours used in a forbidden clique-colouring. \\

We showed that the multicolour EH-conjecture could be reduced to the situation when the large clique uses the same set of colours as the forbidden colouring or maybe one more.  This brings us to the following special cases, in a sense smallest,  for which the EH-conjecture is known to  be true for the number of colours used in the forbidden colouring, but not any more once additional colours are allowed:

\begin{question}
	Does the $2$-edge-colouring of $K_4$ in which each colour class is isomorphic to $P_4$ have the EH-property for $3$ colours?
\end{question}

\begin{question}\label{problem:triangle}
	Does the rainbow triangle have the EH-property for $4$ colours? 
\end{question}

Note that Question \ref{problem:triangle} was formulated by Conlon, Fox, and R\"odl \cite{CFR} in a connection with a hypergraph Ramsey number $R_3(H; 3)$ for a specific $3$-uniform hypergraph $H$. Here $R_3(H; 3)$ is the smallest $n$ such that any colouring of triples from an $n$-element set using three colours results in a monochromatic copy of $H$.

Let $H_t$ be a $3$-uniform hypergraph on  a vertex set $[t]\cup \binom{[t]}{2}$ and edge set  $\{\{i, j, \{i,j\}\}: i, j \in [t]\}$.  It was shown in \cite{CFR}, that 
$$F(t)\leq R_3(H_t; 3) =O(t^4 F(t^3)^2),$$
where $F(t)$ is a dual function for $h_4(n, c')$,  so that $c'$ is an edge colouring of a triangle using three colours $1, 2, $ and $3$, i.e., 
$F(t)$ is the smallest $n$ such that any edge colouring of a clique on $n$ vertices using colours $1, 2, 3$, and $4$ contains either $c'$ or  a clique of size $t$ inducing at most $3$ colours. 
So, if Question \ref{problem:triangle} has a positive answer, it would imply a polynomial behaviour of  $R_3(H_t; 3)$, contrasting the exponential behaviour of  the $4$-colour Ramsey number $R_3(H_t; 4)$.\\

The EH-conjecture fails for $r$-graphs, $r\geq 3$, already when $F$  is a clique of size $r+1$.  Indeed,  well-known results on off-diagonal hypergraph Ramsey numbers show that there are $n$-vertex $r$-graphs that do not have a clique on $r+1$ vertices and do not have cocliques on $f_r(n)$ vertices, where $f_r$ is an iterated logarithmic function (see~\cite{MS18} for the best known results). 
 Moreover, a result (Claim 1.3. in Gishboliner and Tomon~\cite{GT23})   tells us  that for any $r\geq 3$, if $F$ is an $r$-graph on at least $r+1$ vertices, $F\neq D_2$, then there is an $F$-free  $r$-graph  $H$ on $n$ vertices such that $h(H) =(\log n)^{O(1)}$.
Here $D_2$ is a unique $3$-graph on four vertices and two edges.\\

While the  EH-conjecture fails for hypergraphs, we conjecture that its weaker size-version holds.
Here, instead of forbidding a specific  pattern, we forbid a  palette $T=(t_1, \ldots, t_{s'})$, where $t_i$'s add up to the number of edges in a $k$-vertex clique. We say that  a colouring $c$ {\it avoids the palette} $T$ if there exists no clique of size $k$ so that the number of edges of colour $i$ in that clique is exactly $t_i$, $i= 1, \ldots, s'$.\\  

\begin{conjecture} Let $k$, $s'$, and $s$ be integers, $s\geq s'$. Let $T=(t_1, \ldots, t_{s'})$ be a tuple of nonnegative integers adding up to $\binom{k}{2}$. Then there is a positive constant $\epsilon$ such that for any colouring $c$ of $K_n$ in colours from $[s]$  avoiding $T$, we have $h_s(c)\geq n^\epsilon$.
\end{conjecture}

When $s=s'=2$, a quantitative version of the above conjecture is proven in \cite{ABGMW}.
Moreover, one can show that the conjecture holds if $s'=2$ and $s>s'$ using a generalisation of a result by Alon, Pach, and Solymosi \cite{APS} to arbitrary number of  colours, see Weber \cite{W}.  In general, the conjecture still might be challenging as it coincides with the EH-conjecture when $(t_1, \ldots, t_{s'})= (1, \ldots, 1)$, i.e.,  in the case of the rainbow pattern.\\

\noindent
{\bf Acknowledgments}  The authors thank Jacob Fox for bringing their attention to \cite{CFR}.  The research of the first author was supported in part  by the DFG grant FKZ AX 93/2-1.

\bibliographystyle{plain}
\bibliography{bibliography} 
\end{document}